\documentclass[11pt,reqn]{article}

\usepackage[british]{babel}
\usepackage[usenames]{color}
\usepackage{amssymb}
\usepackage{graphics}
\usepackage{amsmath}
\usepackage{amsthm}
\usepackage{amsfonts}
\usepackage{latexsym}
\usepackage{float}

\usepackage[colorlinks=true,
linkcolor=webgreen,
filecolor=webbrown,
citecolor=webgreen]{hyperref}

\definecolor{webgreen}{rgb}{0,.5,0}
\definecolor{webbrown}{rgb}{.6,0,0}

\newcommand{\seqnum}[1]{\href{http://oeis.org/#1}{\underline{#1}}}

\theoremstyle{plain}
\newtheorem{theorem}{Theorem}
\newtheorem{conjecture}[theorem]{Conjecture}
\newtheorem{proposition}[theorem]{Proposition}

\theoremstyle{definition}
\newtheorem{example}[theorem]{Example}
\newtheorem{algorithm}[theorem]{Algorithm}

\begin{document}

\begin{center}
	\textsc{\Large A Space Efficient Algorithm for the Calculation of the Digit Distribution in the Kolakoski Sequence} \vspace{1.5ex}
		
	\textsc{Johan Nilsson}\vspace{1.5ex}
	
	{\small Fakult\"at f\"ur Mathematik, Universit\"at Bielefeld \\
		Postfach 100131, 33501 Bielefeld, Germany }

	\href{mailto:jnilsson@math.uni-bielefeld.de}{\tt jnilsson@math.uni-bielefeld.de}\\
	{\small \today}
	
\end{center}

\begin{abstract}
With standard algorithms for generating the classical Kola\-koski sequence, the numerical calculation of the digit distribution requires a linear amount of space. 
Here, we present an algorithm for calculating the distribution of the digits in the classical Kolakoski sequence, that only requires a logarithmic amount of space and still runs in linear time. The algorithm is easily adaptable to generalised Kolakoski sequences.
\end{abstract}

\section{Introduction}

The classical Kolakoski sequence $K=(K_{n})_{n=1}^{\infty}$ is the unique sequence on the alphabet $\{\mathtt{1},\mathtt{2}\}$ defined as the sequence of its own symbols' run lengths starting with a $\mathtt{1}$. The classical Kolakoski sequence is given in \cite{kolakoski65, kolakoski66}, and is in the On-Line Encyclopedia of Integer Sequences \cite{sloane} with entry number \seqnum{A000002}. The first letters of $K$ are

\begin{equation}
\label{eq: def of kolakoski}
\begin{tabular}{c}
\begin{picture}(265,40)
\put(  0,  0){$K=$}
\put( 30,  0){$\mathtt{1}$}
\put( 46,  0){$\mathtt{2}$}
\put( 62,  0){$\mathtt{2}$}
\put( 78,  0){$\mathtt{1}$}
\put( 94,  0){$\mathtt{1}$}
\put(110,  0){$\mathtt{2}$}
\put(126,  0){$\mathtt{1}$}
\put(142,  0){$\mathtt{2}$}
\put(158,  0){$\mathtt{2}$}
\put(174,  0){$\mathtt{1}$}
\put(190,  0){$\mathtt{2}$}
\put(206,  0){$\mathtt{2}$}
\put(222,  0){$\mathtt{1}$}
\put(238,  0){$\mathtt{1}$}
\put(254,  0){$\ldots$}

\put(  0, 30){$K=$}
\put( 30, 30){$\mathtt{1}$}
\put( 54, 30){$\mathtt{2}$}
\put( 86, 30){$\mathtt{2}$}
\put(110, 30){$\mathtt{1}$}
\put(126, 30){$\mathtt{1}$}
\put(150, 30){$\mathtt{2}$}
\put(174, 30){$\mathtt{1}$}
\put(198, 30){$\mathtt{2}$}
\put(230, 30){$\mathtt{2}$}
\put(246, 30){$\ldots$}

\put( 33, 13){\line(0, 1){ 10}}
\multiput(49, 13)(32, 0){2} {\qbezier(0, 0)(4, 5)( 8,10) \qbezier(16, 0)(12, 5)( 8,10)}
\put(113, 13){\line(0, 1){10}}
\put(129, 13){\line(0, 1){10}}
\multiput(145, 13)(0, 0){1} {\qbezier(0, 0)(4, 5)( 8,10) \qbezier(16, 0)(12, 5)( 8,10)}
\put(177, 13){\line(0, 1){10}}
\multiput(193, 13)(32, 0){2}{\qbezier(0, 0)(4, 5)( 8,10) \qbezier(16, 0)(12, 5)( 8,10)}
\end{picture}
\end{tabular}
\end{equation}
There are several interesting questions, answered and unanswered, on the properties of the classical Kolakoski sequence;  Kimberling presents several of these in  \cite{kimberling}. One of the simplest, and yet unresolved, questions is that of the distribution of digits in $K$. If we let $o_{n}$ be the number of $\mathtt{1}$s in $K$ up to and including position $n$, that is $o_{n} = |\{i: K_{i}=\mathtt{1}, 1\leq i \leq n\}|$, then the conjecture is  

\begin{conjecture} 
\label{conj: distribution}
The limit $\lim_{n\to\infty} \frac{o_{n}}{n}$ exists and equals $\frac{1}{2}$.
\end{conjecture}

Both parts of Conjecture \ref{conj: distribution}, the existence and the value, are still open. Several aspects of the conjecture (along with other properties and questions regarding the Kolakoski sequence as well) are considered by Dekking in \cite{dekking79, dekking80, dekking97}; see also the survey by Sing \cite{sing2010} and further references therein. 

In $\cite{steinsky}$ Steinsky describes a recursion that generates the letters $K_n$ and uses it to numerically calculate the distribution of the $\mathtt{1}$s up to $n = 3 \cdot 10^{8}$. It is worth noting that a straight-forward implementation of Steinsky's recursion leads to an algorithm that either runs in exponential time or requires a linear amount of space. For some time, Steinsky's result raised doubt as to the validity of Conjecture \ref{conj: distribution}, however subsequent work by Monteil \cite{monteil} suggested once again that the conjecture should hold. 
Monteil used a brute force method, requiring linear time and linear space in $n$, to push the calculation to $n=10^{11}$. 

The brute force, or straight-forward, method to find $o_n$ generates a prefix of length $n$ of the sequence $K$, using the intuitive method suggested by (\ref{eq: def of kolakoski}). That is, starting from a suitable initial sequence, we step through and read off the symbols one by one, with each letter telling us what to write in the sequence beneath, and thus what to append to the end of the current sequence.

We present here an algorithm which runs in linear time, yet only requires a logarithmic amount of space to find $o_n$. Using our algorithm, we can easily push the calculation further than the calculation made by Monteil; we present here values of $o_n$ up to $n= 10^{13}$ (Table \ref{table: classical kolakoski}).
Our calculation indicates that Conjecture \ref{conj: distribution} should hold, but once again gives no definite answer. We present our algorithm in Section \ref{sec: algorithm} and state and prove the algorithm's run time performance in Section \ref{sec: analysis}. In Section \ref{sec: general}, we briefly remark on our algorithm's adaptability to more general Kolakoski sequences, and finally in Section \ref{sec: calculations} we present the results of our calculations.

\section{The Algorithm}
\label{sec: algorithm}

We present here an algorithm for calculating the number of $\mathtt{1}$s and $\mathtt{2}$s in the classical Kolakoski sequence $K$ up to a position $n$. Our algorithm is more memory-efficient than the straight-forward algorithm for finding $K_{n}$; it requires only $O(\log n)$ amount of space (Proposition \ref{prop: space}) compared to the $O(n)$ for a brute force algorithm. 
Here we use the standard asymptotic notation $O(n)$. That is, we write $f(n) = O(g(n))$ if there is a constant $c$ such that $f(n)\leq c\,g(n)$ for all $n$. (For more of this see \cite{cormen}.) The run time of our algorithm is $O(n)$ to find $o_n$ (Proposition \ref{prop: time}); this is the same as for the brute force method.

The idea in our algorithm is that if we set out only to find $o_{n}$, we do not have to save the complete sequence up to position $n$ when stepping through the sequence $K$. 
As in the intuitive way of generating $K$, we look back at a previous position to see which symbol run to append.
However, this previous position is itself determined by a letter even further back, and so on. 
If we keep track only of these positions that we ``look back at'', we can drastically reduce the amount of space needed by the algorithm.

To get a hint of how this can be done, we take as a starting point a scheme, as in (\ref{eq: def of kolakoski}). We see that the upper row defines (or conversely, may be defined as) the run lengths of the symbols in the lower one. 
We expand this scheme by adding more rows above and connecting each symbol to the symbol in the row above that has (via run length) generated it. In this way, we obtain a tree structure, as illustrated in Figure \ref{fig: kolakoski tree}.

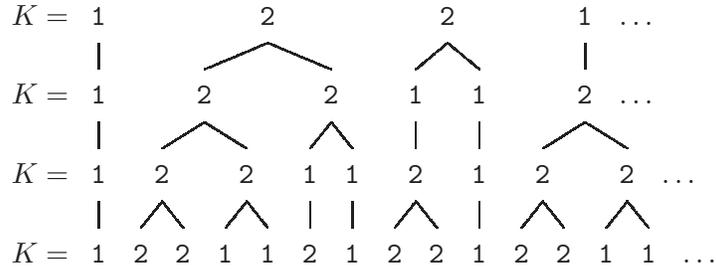
\begin{figure}
\begin{center}
\begin{tabular}{c}
\begin{picture}(265,100)
\put(  0,  0){$K=$}
\put( 30,  0){$\mathtt{1}$}
\put( 46,  0){$\mathtt{2}$}
\put( 62,  0){$\mathtt{2}$}
\put( 78,  0){$\mathtt{1}$}
\put( 94,  0){$\mathtt{1}$}
\put(110,  0){$\mathtt{2}$}
\put(126,  0){$\mathtt{1}$}
\put(142,  0){$\mathtt{2}$}
\put(158,  0){$\mathtt{2}$}
\put(174,  0){$\mathtt{1}$}
\put(190,  0){$\mathtt{2}$}
\put(206,  0){$\mathtt{2}$}
\put(222,  0){$\mathtt{1}$}
\put(238,  0){$\mathtt{1}$}
\put(254,  0){$\ldots$}

\put(  0, 30){$K=$}
\put( 30, 30){$\mathtt{1}$}
\put( 54, 30){$\mathtt{2}$}
\put( 86, 30){$\mathtt{2}$}
\put(110, 30){$\mathtt{1}$}
\put(126, 30){$\mathtt{1}$}
\put(150, 30){$\mathtt{2}$}
\put(174, 30){$\mathtt{1}$}
\put(198, 30){$\mathtt{2}$}
\put(230, 30){$\mathtt{2}$}
\put(246, 30){$\ldots$}

\put(  0, 60){$K=$}
\put( 30, 60){$\mathtt{1}$}
\put( 70, 60){$\mathtt{2}$}
\put(118, 60){$\mathtt{2}$}
\put(150, 60){$\mathtt{1}$}
\put(174, 60){$\mathtt{1}$}
\put(214, 60){$\mathtt{2}$}
\put(230, 60){$\ldots$}

\put(  0, 90){$K=$}
\put( 30, 90){$\mathtt{1}$}
\put( 94, 90){$\mathtt{2}$}
\put(162, 90){$\mathtt{2}$}
\put(214, 90){$\mathtt{1}$}
\put(230, 90){$\ldots$}
\put( 33, 13){\line(0, 1){10}}
\multiput( 49, 13)(32, 0){2}{\qbezier( 0, 0)( 4, 5)( 8,10) \qbezier(16, 0)(12, 5)( 8,10)}
\put(113, 13){\line(0, 1){ 10}}
\put(129, 13){\line(0, 1){ 10}}
\multiput(145, 13)( 0, 0){1}{\qbezier( 0, 0)( 4, 5)( 8,10) \qbezier(16, 0)(12, 5)( 8,10)}
\put(177, 13){\line(0, 1){ 10}}
\multiput(193, 13)(32, 0){2}{\qbezier(0, 0)(4, 5)(8, 10)\qbezier(16, 0)(12, 5)(8, 10)}

\put( 33, 43){\line(0, 1){ 10}}
\multiput( 57, 43)( 0, 0){1}{\qbezier( 0, 0)( 8, 5)(16,10) \qbezier(32, 0)(24, 5)(16,10)}
\multiput(113, 43)( 0, 0){1}{\qbezier( 0, 0)( 4, 5)( 8,10) \qbezier(16, 0)(12, 5)( 8,10)}
\put( 153, 43){\line(0, 1){ 10}}
\put( 177, 43){\line(0, 1){ 10}}
\multiput(201, 43)( 0, 0){1}{\qbezier( 0, 0)( 8, 5)(16,10) \qbezier(32, 0)(24, 5)(16,10)}

\put( 33, 73){\line(0, 1){ 10}}
\multiput( 73, 73)( 0, 0){1}{\qbezier( 0, 0)(12, 5)(24,10) \qbezier(48, 0)(36, 5)(24,10)}
\multiput(153, 73)( 0, 0){1}{\qbezier( 0, 0)( 6, 5)(12,10) \qbezier(24, 0)(18, 5)(12,10)}
\put(217, 73){\line(0, 1){ 10}}

\end{picture}
\end{tabular}
\caption{\label{fig: kolakoski tree}The tree structure in the Kolakoski sequence.}
\end{center}
\end{figure}

We may thus interpret the letters in the classical Kolakoski sequence $K$ as the leaves of a tree, (the leaves are the symbols in the bottom row in Figure \ref{fig: kolakoski tree}). Each internal node in this tree structure is a symbol in in an upper row interpreted as a run length. Each letter is connected to the letter above that has generated it (called an ancestor), and also to the letter(s) below that it generates, termed children. This tree structure continues upwards without bound as we step through the symbols of the Kolakoski sequence. However, we only need to go up in the tree until we find an ancestor, to the leaf we are currently looking at, at a left most position. 

From this point on we shall consider the sequence $K'$, defined by $K= \mathtt{1}K'$. This simplifies matters somewhat, as we do not then have to deal with the left most $\mathtt{1}$s at each height in the tree.  

The algorithm for finding $o_n$ can concisely be described as an  ``in-order traverse'' of this tree structure, where we start from the lower left, and where we keep track of the symbols we see in the leaves during the traverse. While traversing, we add new ancestors when needed; that is we build the tree as we traverse it.
To reduce the memory requirement, we dynamically generate and keep track only of the part of the tree that we currently use for the traverse. While doing so, we store the ancestors along with an indicator that tells us which of its children we have already traversed. To this end, we introduce pointers $P_k$, which are assigned values from the set $S = \{\mathtt{1},\mathtt{2},\mathtt{11},\mathtt{22}\}$. 
Note that here, a run is defined as word from the set $S$.
At any given time, the pointer $P_0$ holds the current run in the leaves and $P_1$ holds the ancestor to $P_0$. Similarly, any $P_k$ that has been initiated holds the ancestor to $P_{k-1}$. 

We say here that pointers ``hold'' and not ``are'' a run because $P_k$ may contain more 
than just the single-symbol ancestor of $P_{k-1}$, it may also contain a sibling of $P_{k}$. Here we refer to the single symbols (that is, $\mathtt{1}$s or $\mathtt{2}$s) of a two symbol run ($\mathtt{11}$ or $\mathtt{22}$) as siblings.

The algorithm can now be described as follows.

\begin{algorithm}
\label{alg: main algorithm}~

\begin{itemize}
\item[-] To increase (or to assign a new value to) the pointer $P_k$ we proceed as follows. Firstly, if $P_k$ has not been initiated, let $P_k = \mathtt{22}$. 
If $P_k$, for $k>0$, contains two symbols then remove one of the symbols in $P_k$; otherwise (if $k=0$), increase $P_1$.

If, on the other hand, $P_k$ contains only one symbol, then increase $P_{k+1}$ recursively. When this increment is done, the new run to write in $P_{k}$ is of the length given by the first symbol in $P_{k+1}$ and the run to write has symbol(s) opposite to the symbol(s) previously held by $P_{k}$. Note that here we do not remove the first symbol of $P_{k+1}$ when we return from the recursion. 

\item[-] To step throw the sequence $K$ (from its second symbol onwards) and calculate $o_n$, we repeatedly increment the pointer $P_0$ and keep track of the number of $\mathtt{1}$s and $\mathtt{2}$s that we see.  
\hfill$\diamond$
\end{itemize}
\end{algorithm}

Note that for a given run contained in $P_0$, the algorithm will generate only the pointers $P_1,\ldots, P_N$ to $P_0$, where the ancestor in $P_N$ is at the left most position in the sequence $K'$. (And it is this height $N$ that we shall shortly show is of the order of $\log n$ when $P_0$ holds the $n$th letter in the sequence).
As we step through the algorithm, we shall see that the successive runs held by the pointer $P_0$ (and also for other $P_k$) are the symbols in the sequence $K'$. 
nI pseudo-code the increment of $P_0$, (or the step by step traverse of the leaves), would be done with the recursive call of the procedure \texttt{IncrementPointer} as presented below.

\small
\begin{verbatim}
// Increments the pointer at height n. 
// After initiating P[0] succesive calls to IncrementPointer(0)
// will yield the Kolakoski sequence from the second term onward.

IncrementPointer(int k)
{   if(P[k] has not been initiated) 
    {   P[k] = 22
    }

    if(k == 0)
    {   IncrementPointer(1)
        if(P[0] == 1 or P[0] == 11)
        {   P[0] = (P[1] == 1) ? 2 : 22
        }else
        {   P[0] = (P[1] == 1) ? 1 : 11
        }
    }else if(P[k] == 1)
    {   IncrementPointer(k+1)
        P[k] = (P[k+1] == 1 or P[k+1] == 11) ? 2 : 22
    }else if(P[k] == 2)
    {   IncrementPointer(k+1)
        P[k] = (P[k+1] == 1 or P[k+1] == 11) ? 1 : 11
    }else if(P[k] == 11)
    {   P[k] = 1
    }else
    {   P[k] = 2
    }
}
\end{verbatim}
\normalsize

To illustrate how the algorithm works, we now present through of its initial steps.

\begin{example}
Incrementing the pointer $P_{0}$ once is done through the following procedure;

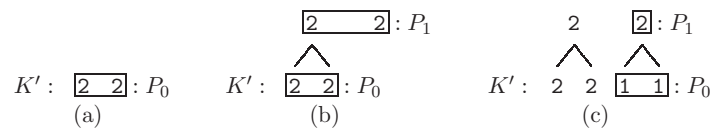
\begin{figure}[H]
\begin{center}
\scalebox{0.8}{
\begin{tabular}{c@{\hspace{30pt}}c@{\hspace{30pt}}c}
\begin{picture}(70,50)
\put(  0,  0){$K':$}
\put( 30,  0){$\mathtt{2}$}
\put( 46,  0){$\mathtt{2}$}
\put( 29, -1){\framebox(24,10){}}
\put( 56,  0){$:P_0$}
\end{picture}
&
\begin{picture}(94,50)
\put(  0,  0){$K':$}
\put( 30,  0){$\mathtt{2}$}
\put( 46,  0){$\mathtt{2}$}
\put( 29, -1){\framebox(24,10){}}
\put( 56,  0){$:P_0$}
\put( 38, 30){$\mathtt{2}$}
\put( 70, 30){$\mathtt{2}$}
\put( 37, 29){\framebox(40,10){}}
\put( 80, 30){$:P_1$}
\multiput(33, 13)(32, 0){1}{\qbezier(0, 0)(4, 5)(8,10)\qbezier(16, 0)(12, 5)(8,10)}
\end{picture}
&
\begin{picture}(102,50)
\put(  0,  0){$K':$}
\put( 30,  0){$\mathtt{2}$}
\put( 46,  0){$\mathtt{2}$}
\put( 62,  0){$\mathtt{1}$}
\put( 78,  0){$\mathtt{1}$}
\put( 61, -1){\framebox(24,10){}}
\put( 88,  0){$:P_0$}
\put( 38,  30){$\mathtt{2}$}
\put( 70,  30){$\mathtt{2}$}
\put( 69, 29){\framebox( 8,10){}}
\put( 80, 30){$:P_1$}
\multiput(33, 13)(32, 0){2}{\qbezier(0, 0)(4, 5)(8,10) \qbezier(16, 0)(12, 5)(8,10)}
\end{picture}
\\
(a) & (b) &(c)
\end{tabular}
}
\end{center}
\caption{\label{fig: first}The first increment of the pointer $P_0$.}
\end{figure}

Figure \ref{fig: first} illustrates the first increment of the pointer $P_0$ in the algorithm. 
(a) The initiation of $P_0$. The framed symbols $\mathtt{22}$ are the contents of the pointer $P_0$. 
(b) To continue our leaf traverse we must generate the next leaf. This is done by looking at the ancestor of the run held by $P_0$. As this ancestor does not exist we have to generate it, that is we set $P_1 = \mathtt{22}$. 
(c) The first symbol of $P_1$ already has children (that is, it generated the initial run held by $P_0$). 
Therefore we step to the second symbol of $P_1$. The new run to assign to $P_0$ (that is, the new leaf we traverse) is then $\mathtt{11}$, since the current symbol in $P_1$ is $\mathtt{2}$ and $P_0$ currently holds the run $\mathtt{22}$.

\begin{figure}[H]
\begin{center}
\scalebox{0.8}{
\begin{tabular}{c@{\hspace{30pt}}c@{\hspace{30pt}}c}
\begin{picture}(110,80)
\put(  0,  0){$K':$}
\put( 30,  0){$\mathtt{2}$}
\put( 46,  0){$\mathtt{2}$}
\put( 62,  0){$\mathtt{1}$}
\put( 78,  0){$\mathtt{1}$}
\put( 61, -1){\framebox(24,10){}}
\put( 88,  0){$:P_0$}
\put( 38,  30){$\mathtt{2}$}
\put( 70,  30){$\mathtt{2}$}
\put( 69, 29){\framebox( 8,10){}}
\put( 80, 30){$:P_1$}
\put( 54, 60){$\mathtt{2}$}
\put(102, 60){$\mathtt{2}$}
\put( 53, 59){\framebox( 56,10){}}
\put(112, 60){$:P_2$}
\multiput(33, 13)(32, 0){2}{\qbezier(0, 0)(4, 5)( 8,10) \qbezier(16, 0)(12, 5)( 8,10)}
\multiput(41, 43)( 0, 0){1}{\qbezier(0, 0)(8, 5)(16,10) \qbezier(32, 0)(24, 5)(16,10)}
\end{picture}
&
\begin{picture}(124,80)
\put(  0,  0){$K':$}
\put( 30,  0){$\mathtt{2}$}
\put( 46,  0){$\mathtt{2}$}
\put( 62,  0){$\mathtt{1}$}
\put( 78,  0){$\mathtt{1}$}
\put( 61, -1){\framebox(24,10){}}
\put( 88,  0){$:P_0$}
\put( 38, 30){$\mathtt{2}$}
\put( 70, 30){$\mathtt{2}$}
\put( 94, 30){$\mathtt{1}$}
\put(110, 30){$\mathtt{1}$}
\put( 93, 29){\framebox(24,10){}}
\put(120, 30){$:P_1$}
\put( 54, 60){$\mathtt{2}$}
\put(102, 60){$\mathtt{2}$}
\put(101, 59){\framebox( 8,10){}}
\put(112, 60){$:P_2$}
\multiput(33, 13)(32, 0){2}{\qbezier(0, 0)(4, 5)( 8,10) \qbezier(16, 0)(12, 5)( 8,10)}
\multiput(41, 43)( 0, 0){1}{\qbezier(0, 0)(8, 5)(16,10) \qbezier(32, 0)(24, 5)(16,10)}
\multiput(97, 43)( 0, 0){1}{\qbezier(0, 0)(4, 5)( 8,10) \qbezier(16, 0)(12, 5)( 8,10)}
\end{picture}
&
\begin{picture}(124,80)
\put(  0,  0){$K':$}
\put( 30,  0){$\mathtt{2}$}
\put( 46,  0){$\mathtt{2}$}
\put( 62,  0){$\mathtt{1}$}
\put( 78,  0){$\mathtt{1}$}
\put( 94,  0){$\mathtt{2}$}
\put( 93, -1){\framebox(8,10){}}
\put(104,  0){$:P_0$}
\put( 38, 30){$\mathtt{2}$}
\put( 70, 30){$\mathtt{2}$}
\put( 94, 30){$\mathtt{1}$}
\put(110, 30){$\mathtt{1}$}
\put( 93, 29){\framebox(24,10){}}
\put(120, 30){$:P_1$}
\put( 54, 60){$\mathtt{2}$}
\put(102, 60){$\mathtt{2}$}
\put(101, 59){\framebox( 8,10){}}
\put(112, 60){$:P_2$}
\multiput(33, 13)(32, 0){2}{\qbezier(0, 0)(4, 5)( 8,10) \qbezier(16, 0)(12, 5)( 8,10)}
\multiput(41, 43)( 0, 0){1}{\qbezier(0, 0)(8, 5)(16,10) \qbezier(32, 0)(24, 5)(16,10)}
\multiput(97, 43)( 0, 0){1}{\qbezier(0, 0)(4, 5)( 8,10) \qbezier(16, 0)(12, 5)( 8,10)}
\put( 97, 13){\line(0, 1){10}}
\end{picture}
\\
(a) & (b) & (c)
\end{tabular}
}
\caption{\label{fig: second}The second increment of the pointer $P_0$.}
\end{center}
\end{figure}
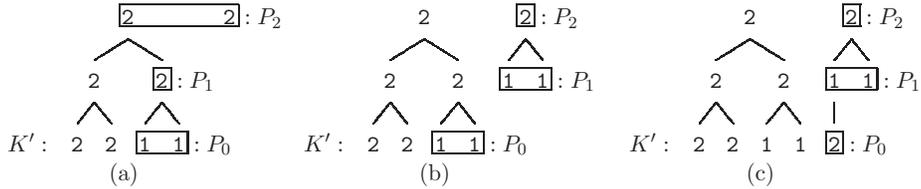

Figure \ref{fig: second} illustrates the second increment of the pointer $P_0$ in the algorithm. 
(a) To generate the next leaf we have to look at the ancestor of the run currently held by $P_0$. That is, we look at the pointer $P_1$. 
But since we have already used the symbol in $P_1$ we have to recursively look at the ancestor of $P_1$. This does not exist, so we initiate the ancestor and pointer $P_2 = \mathtt{22}$. 
(b) As the first symbol of $P_2$ already has children, we step to its second symbol. The new run to assign to $P_1$ is then $\mathtt{11}$, since the relevant ancestor in $P_2$ is $\mathtt{2}$ and $P_1$ currently holds the run $\mathtt{22}$. 
(c) We have not yet generated any of the children of any of the symbols held by $P_1$ and therefore the current one is the first one. This provides the new run of $\mathtt{2}$ in $P_0$, since the first symbol in $P_1$ is $\mathtt{1}$ and $P_0$ currently holds $\mathtt{11}$.

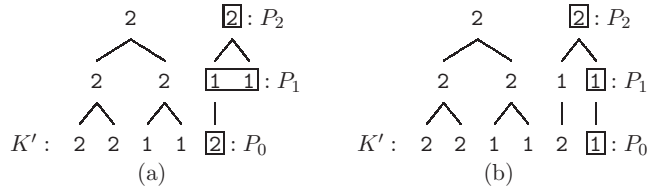
\begin{figure}[H]
\begin{center}
\scalebox{0.8}{
\begin{tabular}{c@{\hspace{30pt}}c@{\hspace{30pt}}c}
\begin{picture}(134,80)
\put(  0,  0){$K':$}
\put( 30,  0){$\mathtt{2}$}
\put( 46,  0){$\mathtt{2}$}
\put( 62,  0){$\mathtt{1}$}
\put( 78,  0){$\mathtt{1}$}
\put( 94,  0){$\mathtt{2}$}
\put( 93, -1){\framebox(8,10){}}
\put(104,  0){$:P_0$}
\put( 38, 30){$\mathtt{2}$}
\put( 70, 30){$\mathtt{2}$}
\put( 94, 30){$\mathtt{1}$}
\put(110, 30){$\mathtt{1}$}
\put( 93, 29){\framebox(24,10){}}
\put(120, 30){$:P_1$}
\put( 54, 60){$\mathtt{2}$}
\put(102, 60){$\mathtt{2}$}
\put(101, 59){\framebox( 8,10){}}
\put(112, 60){$:P_2$}
\multiput(33, 13)(32, 0){2}{\qbezier(0, 0)(4, 5)( 8,10) \qbezier(16, 0)(12, 5)( 8,10)}
\multiput(41, 43)( 0, 0){1}{\qbezier(0, 0)(8, 5)(16,10) \qbezier(32, 0)(24, 5)(16,10)}
\multiput(97, 43)( 0, 0){1}{\qbezier(0, 0)(4, 5)( 8,10) \qbezier(16, 0)(12, 5)( 8,10)}
\put( 97, 13){\line(0, 1){10}}
\end{picture}
&
\begin{picture}(134,80)
\put(  0,  0){$K':$}
\put( 30,  0){$\mathtt{2}$}
\put( 46,  0){$\mathtt{2}$}
\put( 62,  0){$\mathtt{1}$}
\put( 78,  0){$\mathtt{1}$}
\put( 94,  0){$\mathtt{2}$}
\put(110,  0){$\mathtt{1}$}
\put(109, -1){\framebox(8,10){}}
\put(120,  0){$:P_0$}
\put( 38, 30){$\mathtt{2}$}
\put( 70, 30){$\mathtt{2}$}
\put( 94, 30){$\mathtt{1}$}
\put(110, 30){$\mathtt{1}$}
\put(109, 29){\framebox( 8,10){}}
\put(120, 30){$:P_1$}
\put( 54, 60){$\mathtt{2}$}
\put(102, 60){$\mathtt{2}$}
\put(101, 59){\framebox( 8,10){}}
\put(112, 60){$:P_2$}
\multiput(33, 13)(32, 0){2}{\qbezier(0, 0)( 4, 5)( 8,10) \qbezier(16, 0)(12, 5)( 8,10)}
\multiput(41, 43)( 0, 0){1}{\qbezier(0, 0)( 8, 5)(16,10) \qbezier(32, 0)(24, 5)(16,10)}
\multiput(97, 43)( 0, 0){1}{\qbezier(0, 0)( 4, 5)( 8,10) \qbezier(16, 0)(12, 5)( 8,10)}
\put( 97, 13){\line(0, 1){10}}
\put(113, 13){\line(0, 1){10}}
\end{picture}
\\
(a) & (b)
\end{tabular}
}
\caption{\label{fig: third}The third increment of the pointer $P_0$.}
\end{center}
\end{figure}

Figure \ref{fig: third} illustrates the third increment of the pointer $P_0$ in the algorithm.
(a) The status of the pointers after the second increment of $P_0$. Note that we have only used the first symbol held by $P_1$. 
(b) To generate the next leaf we look at the ancestor of the run currently held by $P_0$, that is $P_1$, which contains the run $\mathtt{11}$. The first symbol already has a child, so we use the second symbol, $\mathtt{1}$, to generate the new run in $P_0$, which is $\mathtt{2}$, as $P_0$ currently holds the run $\mathtt{1}$.

\begin{figure}[H]
\begin{center}
\scalebox{0.8}{
\begin{tabular}{c@{\hspace{30pt}}c}
\begin{picture}(168,110)
\put(  0,  0){$K':$}
\put( 30,  0){$\mathtt{2}$}
\put( 46,  0){$\mathtt{2}$}
\put( 62,  0){$\mathtt{1}$}
\put( 78,  0){$\mathtt{1}$}
\put( 94,  0){$\mathtt{2}$}
\put(110,  0){$\mathtt{1}$}
\put(109, -1){\framebox(8,10){}}
\put(120,  0){$:P_0$}
\put( 38, 30){$\mathtt{2}$}
\put( 70, 30){$\mathtt{2}$}
\put( 94, 30){$\mathtt{1}$}
\put(110, 30){$\mathtt{1}$}
\put(109, 29){\framebox( 8,10){}}
\put(120, 30){$:P_1$}
\put( 54, 60){$\mathtt{2}$}
\put(102, 60){$\mathtt{2}$}
\put(101, 59){\framebox( 8,10){}}
\put(112, 60){$:P_2$}
\put( 78, 90){$\mathtt{2}$}
\put(146, 90){$\mathtt{2}$}
\put( 77, 89){\framebox(76,10){}}
\put(156, 90){$:P_3$}
\multiput(33, 13)(32, 0){2}{\qbezier(0, 0)(4, 5)( 8, 10) \qbezier(16, 0)(12, 5)( 8,10)}
\multiput(41, 43)( 0, 0){1}{\qbezier(0, 0)(8, 5)(16, 10) \qbezier(32, 0)(24, 5)(16,10)}
\multiput(97, 43)( 0, 0){1}{\qbezier(0, 0)(4, 5)( 8, 10) \qbezier(16, 0)(12, 5)( 8,10)}
\put( 97, 13){\line(0, 1){10}}
\put(113, 13){\line(0, 1){10}}
\multiput(57, 73)(0, 0){1}{\qbezier( 0, 0)(12, 5)(24,10) \qbezier(48, 0)(36, 5)(24,10)}
\end{picture}
&
\begin{picture}(180,110)
\put(  0,  0){$K':$}
\put( 30,  0){$\mathtt{2}$}
\put( 46,  0){$\mathtt{2}$}
\put( 62,  0){$\mathtt{1}$}
\put( 78,  0){$\mathtt{1}$}
\put( 94,  0){$\mathtt{2}$}
\put(110,  0){$\mathtt{1}$}
\put(109, -1){\framebox(8,10){}}
\put(120,  0){$:P_0$}
\put( 38, 30){$\mathtt{2}$}
\put( 70, 30){$\mathtt{2}$}
\put( 94, 30){$\mathtt{1}$}
\put(110, 30){$\mathtt{1}$}
\put(109, 29){\framebox( 8,10){}}
\put(120, 30){$:P_1$}
\put( 54, 60){$\mathtt{2}$}
\put(102, 60){$\mathtt{2}$}
\put(134, 60){$\mathtt{1}$}
\put(158, 60){$\mathtt{1}$}
\put(133, 59){\framebox(32,10){}}
\put(168, 60){$:P_2$}
\put( 78, 90){$\mathtt{2}$}
\put(146, 90){$\mathtt{2}$}
\put(145, 89){\framebox( 8,10){}}
\put(156, 90){$:P_3$}
\multiput(33, 13)(32, 0){2}{\qbezier(0, 0)( 4, 5)( 8,10) \qbezier(16, 0)(12, 5)( 8,10)}
\multiput(41, 43)( 0, 0){1}{\qbezier(0, 0)( 8, 5)(16,10) \qbezier(32, 0)(24, 5)(16,10)}
\multiput(97, 43)( 0, 0){1}{\qbezier(0, 0)( 4, 5)( 8,10) \qbezier(16, 0)(12, 5)( 8,10)}
\put( 97, 13){\line(0, 1){10}}
\put(113, 13){\line(0, 1){10}}
\multiput( 57, 73)(32, 0){1}{\qbezier( 0, 0)(12, 5)(24,10) \qbezier(48, 0)(36, 5)(24,10)}
\multiput(137, 73)(32, 0){1}{\qbezier( 0, 0)( 6, 5)(12,10) \qbezier(24, 0)(18, 5)(12,10)}
\end{picture}
\\
(a) & (b)
\end{tabular}
}
\caption{\label{fig: fourth first}The first part of the fourth increment of the pointer $P_0$.}
\end{center}
\end{figure}
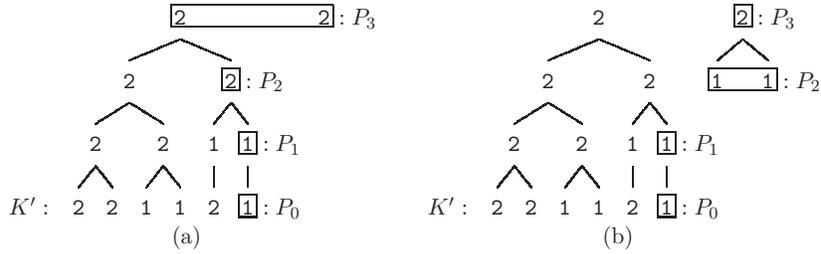

Figure \ref{fig: fourth first} illustrates the first part of the fourth increment of the pointer $P_0$ in the algorithm.
(a) To increase $P_0$ we have to look at the ancestors of the run held by $P_0$. We see that we have used all symbols in all of the ancestors, therefore we have to initiate  the new pointer $P_3=\mathtt{22}$.
(b) We have already used the first symbol held by $P_3$ and therefore we step to its second symbol. The new run to assign to $P_2$ is now $\mathtt{11}$ since  $P_3=\mathtt{2}$ and $P_2= \mathtt{22}$.

\begin{figure}[H]
\begin{center}
\scalebox{0.8}{
\begin{tabular}{c@{\hspace{30pt}}c@{\hspace{30pt}}c}
\begin{picture}(180,110)
\put(  0,  0){$K':$}
\put( 30,  0){$\mathtt{2}$}
\put( 46,  0){$\mathtt{2}$}
\put( 62,  0){$\mathtt{1}$}
\put( 78,  0){$\mathtt{1}$}
\put( 94,  0){$\mathtt{2}$}
\put(110,  0){$\mathtt{1}$}
\put(109, -1){\framebox(8,10){}}
\put(120,  0){$:P_0$}
\put( 38, 30){$\mathtt{2}$}
\put( 70, 30){$\mathtt{2}$}
\put( 94, 30){$\mathtt{1}$}
\put(110, 30){$\mathtt{1}$}
\put(134, 30){$\mathtt{2}$}
\put(133, 29){\framebox( 8,10){}}
\put(144, 30){$:P_1$}
\put( 54, 60){$\mathtt{2}$}
\put(102, 60){$\mathtt{2}$}
\put(134, 60){$\mathtt{1}$}
\put(158, 60){$\mathtt{1}$}
\put(133, 59){\framebox(32,10){}}
\put(168, 60){$:P_2$}
\put( 78, 90){$\mathtt{2}$}
\put(146, 90){$\mathtt{2}$}
\put(145, 89){\framebox( 8,10){}}
\put(156, 90){$:P_3$}
\multiput( 33, 13)(32, 0){2}{\qbezier( 0, 0)( 4, 5)( 8,10) \qbezier(16, 0)(12, 5)( 8,10)}
\multiput( 41, 43)( 0, 0){1}{\qbezier( 0, 0)( 8, 5)(16,10) \qbezier(32, 0)(24, 5)(16,10)}
\multiput( 97, 43)( 0, 0){1}{\qbezier( 0, 0)( 4, 5)( 8,10) \qbezier(16, 0)(12, 5)( 8,10)}
\put( 97, 13){\line(0, 1){10}}
\put(113, 13){\line(0, 1){10}}
\multiput( 57, 73)(32, 0){1}{\qbezier( 0, 0)(12, 5)(24,10) \qbezier(48, 0)(36, 5)(24,10)}
\multiput(137, 73)(32, 0){1}{\qbezier( 0, 0)( 6, 5)(12,10) \qbezier(24, 0)(18, 5)(12,10)}
\put(137, 43){\line(0, 1){ 10}}
\end{picture}
&
\begin{picture}(180,110)
\put(  0,  0){$K':$}
\put( 30,  0){$\mathtt{2}$}
\put( 46,  0){$\mathtt{2}$}
\put( 62,  0){$\mathtt{1}$}
\put( 78,  0){$\mathtt{1}$}
\put( 94,  0){$\mathtt{2}$}
\put(110,  0){$\mathtt{1}$}
\put(126,  0){$\mathtt{2}$}
\put(142,  0){$\mathtt{2}$}
\put(125, -1){\framebox(24,10){}}
\put(152,  0){$:P_0$}
\put( 38, 30){$\mathtt{2}$}
\put( 70, 30){$\mathtt{2}$}
\put( 94, 30){$\mathtt{1}$}
\put(110, 30){$\mathtt{1}$}
\put(134, 30){$\mathtt{2}$}
\put(133, 29){\framebox( 8,10){}}
\put(144, 30){$:P_1$}
\put( 54, 60){$\mathtt{2}$}
\put(102, 60){$\mathtt{2}$}
\put(134, 60){$\mathtt{1}$}
\put(158, 60){$\mathtt{1}$}
\put(133, 59){\framebox(32,10){}}
\put(168, 60){$:P_2$}
\put( 78, 90){$\mathtt{2}$}
\put(146, 90){$\mathtt{2}$}
\put(145, 89){\framebox( 8,10){}}
\put(156, 90){$:P_3$}
\multiput( 33, 13)(32, 0){2}{\qbezier( 0, 0)( 4, 5)( 8,10) \qbezier(16, 0)(12, 5)( 8,10)}
\put( 97, 13){\line(0, 1){ 10}}
\put(113, 13){\line(0, 1){ 10}}
\multiput(129, 13)( 0, 0){1}{\qbezier( 0, 0)( 4, 5)( 8,10) \qbezier(16, 0)(12, 5)( 8,10)}
\multiput( 41, 43)( 0, 0){1}{\qbezier( 0, 0)( 8, 5)(16,10) \qbezier(32, 0)(24, 5)(16,10)}
\multiput( 97, 43)(32, 0){1}{\qbezier( 0, 0)( 4, 5)( 8,10) \qbezier(16, 0)(12, 5)( 8,10)}
\put(137, 43){\line(0, 1){ 10}}
\multiput( 57, 73)( 0, 0){1}{\qbezier( 0, 0)(12, 5)(24,10) \qbezier(48, 0)(36, 5)(24,10)}
\multiput(137, 73)( 0, 0){1}{\qbezier( 0, 0)( 6, 5)(12,10) \qbezier(24, 0)(18, 5)(12,10)}
\end{picture}
\\
(a) & (b)
\end{tabular}
}
\caption{\label{fig: fourth second} The second part of the fourth increment of the pointer $P_0$.}
\end{center}
\end{figure}
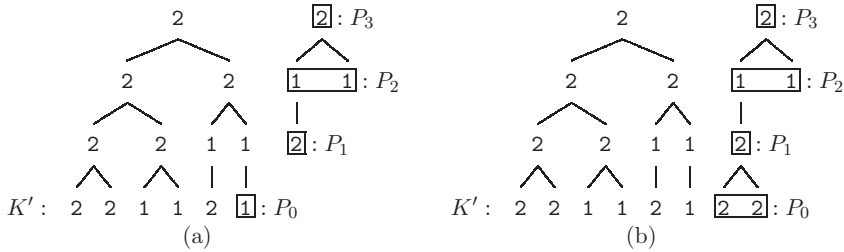

Figure \ref{fig: fourth second} illustrates the second part of the fourth increment of the pointer $P_0$ in the algorithm. 
(a) We have not yet used any of the symbols held by $P_2$ and therefore the current one is the first. Then the new symbol in $P_1$ is $\mathtt{2}$ since the current symbol in $P_2$ is $\mathtt{1}$ and $P_1=\mathtt{1}$.
(b) The new run to assign to $P_0$ is now $\mathtt{22}$ since the first symbol in $P_1$ is $\mathtt{2}$ and $P_0=\mathtt{1}$. 
\qed
\end{example}

Note that the algorithm does not need to keep track of the tree structure that it steps through. The algorithm only keeps track of the current contents of the pointers $P_k$ and how many of each symbol we have seen in $P_0$.

\section{Run Time Analysis of the Algorithm}
\label{sec: analysis}

Let $t_n$ be the number of $\mathtt{2}$s in $K$ up to and including position $n$. That is $t_{n} = |\{i: K_{i}=\mathtt{1}, 1\leq i \leq n\}|$. Recall that we have already similarly defined $o_n$ as the number of ones. By considering words of the form 
\[	\mathtt{11211} 
	\quad \textnormal{and} \quad
	\mathtt{22122}
\]
we see that we have the bounds 
\begin{equation}
\label{eq: frequency bound}
	\frac{1}{4} \leq \frac{o_n}{t_n} \leq 4
\end{equation}
for $n\geq 2$. For the analysis, let $P(n)$ be the number of pointers used by Algorithm \ref{alg: main algorithm} to calculate $o
_n$.

\begin{proposition}
\label{prop: space}
The amount of space used by Algorithm \ref{alg: main algorithm} to find $o_n$ is logarithmic in $n$. That is, $P(n) = O(\log n)$.
\end{proposition}

\begin{proof}
Let $w_0 = \mathtt{122}$ and $w_1 = \mathtt{12211}$ and similarly let $w_k$ be the run length sequence defining $w_{k+1}$. Then $w_k$ is a prefix of the sequence $K$ for all $k\geq0$. (The collection of the words $w_k$ is known as the Kolakoski fan.)  By the frequency bound (\ref{eq: frequency bound}) it follows that 
\[	\frac{6}{5} \leq \frac{|w_{k+1}|}{|w_{k}|} \leq \frac{9}{5} 
\]
whenever $k\geq 1$ and where $|\cdot|$ denotes the length of a word. 

This implies that if pointer $P_0$ holds the symbol at position $n$ in $K'$ then the pointer $P_1$ is at most at position $\frac{5}{6}n$ and at least at position $\frac{5}{9}n$ in the Kolakoski sequence. This argument can now be applied to all pointers. Therefore we see that we have a bound on the number of pointers
\[	P(n) \leq \left\lceil \frac{1}{\log \frac{6}{5}} \, \log n \right\rceil = O(\log n), 
\]
which completes the proof.
\end{proof}

If Conjecture \ref{conj: distribution} were shown to be true, it would follow that the number of pointers needed to find $o_n$ is $P(n) \approx \frac{1}{\log \frac{3}{2}} \, \log n$.

\begin{proposition}
\label{prop: time}
The Algorithm \ref{alg: main algorithm} runs in (amortized) linear time. That is, to find $o_n$ we have to do an amount of work of order $O(n)$.
\end{proposition}

\begin{proof}
Let us consider the maximal amount of work we have to do to make $n$ increments of the pointer $P_0$ (to generate $n$ runs). Note that making $n$ increments of $P_0$ will actually be enough to find at least $o_{\frac{6}{5}n}$, since in each step we generate a run of one or two symbols. Hence, as we seek a maximum, and including the factor $\frac{6}{5}$ would decrease the calculated amount of work by a constant factor, we may simplify our calculation by disregarding this factor.

Let $p_k(n)$ be the number of times we change the contents of pointer $P_k$ under these $n$ increments. Then the sum of the $p_k$s will give us the total amount of work we have to do. It is clear that $p_0(n) = n$, since we change $P_0$ at each increment, and from the algorithm we see directly that $p_1(n) = n$. The other pointers do not change every time; for $k\geq2$ we make a change to $P_k$ only when $P_{k-1}$ consists of a single symbol. 

Let $a_k(n)$ be the number of times the pointer $P_k$ holds a single symbol under $n$ increments of $P_0$. Similarly let $b_k(n)$ be the number of times that $P_k$ holds two symbols under the $n$ increments of $P_0$. 
From the algorithm we see that to find the maximal amount of work, we have to look for the maximal number of single-symbol pointer contents, since this is what forces us to go recursively higher in the tree.
For the pointer $P_0$ it follows from (\ref{eq: frequency bound}) that we have the bounds
\[	\frac{1}{4} \leq \frac{a_0(n)}{b_0(n)} \leq 4
\]
for $n\geq 1$. For pointers higher up, we have that the number of times $P_k$ holds a single symbol is at most four times the number of times it holds two symbols plus the number of times it holds two symbols, since in the latter case $P_k$ will hold a single symbol in the next step of the algorithm. 
This gives
\[	a_k(n) \leq 4 b_k(n) + b_k(n) = 5 b_n(k)
\]
Therefore, our upper bound on the number of times a pointer holds a single symbol gives the bound on the amount of work we have to do with a pointer $P_k$ compared to the amount of work for the pointer holding the children of $P_k$. This is 
\begin{equation}
\label{eq: bound on work between pointers}
	p_{k+1}(n) \leq \frac{5}{6}\, p_k(n) 
\end{equation}
for $k\geq1$. The total amount of work we now have to do to increment the pointer $P_0$ $n$ times is therefore bounded by the initial amount of work plus the convergent geometric series obtained from (\ref{eq: bound on work between pointers})
We have
\begin{equation}
\label{eq: sum of work}
P(n) =  \sum_{i=0}^{\infty}p_i(n) \leq  n + c\log n + n\sum_{i=0}^{\infty} \left(\frac{5}{6}\right)^i \leq (7+C)n,  
\end{equation}
where $c\log n$ is the initial amount of work for each pointer before we can apply our estimates above. 
\end{proof}

\section{Generalised Kolakoski Sequences}
\label{sec: general}

In this this section we remark that our algorithm is also applicable to a general Kolakoski sequence. By a generalised Kolakoski sequence we mean a sequence that is defined as its symbols' run length, as for the classical Kolakoski sequence, but the symbols may be taken from any alphabet $\{r, s\}$, where $r$ and $s$ are natural numbers, as discussed in \cite{dekking80}. We denote a generalised Kolakoski sequence over $r$ and $s$ with $K(r,s)$ and shall assume that $K(r,s)$ starts with the symbol $r$. The classical Kolakoski sequence is then $K = K(\mathtt{1},\mathtt{2})$. 

It is known that if $r+s$ is an even number, then the letter frequency in $K(r,s)$ can be calculated; see \cite{baake, sing2002, sing2003, sing2010}. When $r+s$ is odd, the existence and the value of the letter frequencies are still unknown, but are believed to exist and equal $\frac{1}{2}$. 

Our algorithm easily adopts to count the letters in a generalised Kola\-koski sequence; we may only have to change the initiation of new pointers. By applying the same idea as in the proof of Proposition \ref{prop: space} we see that the algorithm in this case with a generalised Kolakoski sequence uses fewer pointers than for the classical Kolakoski sequence, and therefore the space requirement must again be at most logarithmic. 

Similarly, by looking at the proof of Proposition \ref{prop: time} we see that the number of times we use a pointer for a general Kolakoski sequence before having to consider its ancestor is longer than for the classical Kolakoski sequence. 
Therefore the bounding factor for the quoted amount of work between two consecutive levels (\ref{eq: bound on work between pointers}) must be smaller than the $\frac{5}{6}$ given for the classical Kolakoski sequence. This gives then, by summing up as in (\ref{eq: sum of work}), that the total amount of work for the generalised Kolakoski sequence is also linear in $n$.

\section{Calculations}
\label{sec: calculations}

In Table \ref{table: classical kolakoski} we present a short output from an implementation in Java of our Algorithm \ref{alg: main algorithm} for calculating the number of $\mathtt{1}$s in the classical Kolakoski sequence. The program  was run on a standard PC. In Table \ref{table: generalised kolakoski} we present results of a calculation of the number of $\mathtt{2}$s in the generalised Kolakoski sequence $K(\mathtt{2}, \mathtt{3})$, the sequence \seqnum{A071820} in the On-Line Encyclopedia of Integer Sequences \cite{sloane}. 

We denote for the classical Kolakoski sequence the maximal deviation of the proportion of $\mathtt{1}$s from $\frac{1}{2}$ in a logarithmic decade by 
\[	D(n) = \max_{\frac{1}{10}n < i \leq n }\left|\frac{1}{2}-\frac{o_i}{i}\right|,
\]
where $o_i$ is the number of $\mathtt{1}$s up to position $i$. We can similarly define the deviation for the generalised Kolakoski sequence $K(\mathtt{2}, \mathtt{3})$.

\begin{table}[ht]
\begin{center}
\texttt{
\renewcommand{\arraystretch}{1.2}
\begin{tabular}{|*{5}{r|}}
\hline 
\multicolumn{1}{|c|}{$n$} &{\textnormal{Number of $\mathtt{1}$s}}& \multicolumn{1}{|c|}{ $P(n)$ } & \multicolumn{1}{|c|}{$D(n)$} \\ \hline 
     1 &                      1 &    &  \\ \hline
   $\mathtt{10^{ 1}}$ &                      5 &  4 & $\mathtt{1.667\cdot 10^{-1}}$ \\ \hline
   $\mathtt{10^{ 2}}$ &                     49 & 10 & $\mathtt{8.333\cdot 10^{-2}}$ \\ \hline
   $\mathtt{10^{ 3}}$ &                    502 & 16 & $\mathtt{1.351\cdot 10^{-2}}$ \\ \hline
   $\mathtt{10^{ 4}}$ &                 4\,996 & 22 & $\mathtt{3.588\cdot 10^{-3}}$ \\ \hline
   $\mathtt{10^{ 5}}$ &                49\,972 & 27 & $\mathtt{5.481\cdot 10^{-4}}$ \\ \hline   
   $\mathtt{10^{ 6}}$ &               499\,986 & 33 & $\mathtt{2.800\cdot 10^{-4}}$ \\ \hline   
   $\mathtt{10^{ 7}}$ &            5\,000\,046 & 39 & $\mathtt{3.892\cdot 10^{-5}}$ \\ \hline
   $\mathtt{10^{ 8}}$ &           50\,000\,675 & 44 & $\mathtt{2.054\cdot 10^{-5}}$ \\ \hline
   $\mathtt{10^{ 9}}$ &          500\,001\,223 & 50 & $\mathtt{8.586\cdot 10^{-6}}$ \\ \hline
   $\mathtt{10^{10}}$ &       4\,999\,997\,671 & 56 & $\mathtt{2.152\cdot 10^{-6}}$ \\ \hline
   $\mathtt{10^{11}}$ &      50\,000\,001\,587 & 61 & $\mathtt{4.453\cdot 10^{-7}}$ \\ \hline
   $\mathtt{10^{12}}$ &     500\,000\,050\,701 & 67 & $\mathtt{2.140\cdot 10^{-7}}$ \\ \hline
   $\mathtt{10^{13}}$ &  5\,000\,000\,008\,159 & 73 & $\mathtt{6.774\cdot 10^{-8}}$ \\ \hline
\end{tabular}}
\caption{\label{table: classical kolakoski}The output of the calculation of the number of $\mathtt{1}$s in the classical Kolakoski sequence. Our result here exceeds the calculations made by Steinsky up to $n = 3\cdot 10^8$, and Monteil up to $n=10^{11}$, as mentioned in the introduction.  The column with the number of $\mathtt{1}$s is the sequence \seqnum{A195206} in the On-Line Encyclopedia of Integer Sequences \cite{sloane}.}
\end{center}
\end{table}

\begin{table}[ht]
\begin{center}
\texttt{
\renewcommand{\arraystretch}{1.2}
\begin{tabular}{|*{5}{r|}}
\hline 
\multicolumn{1}{|c|}{$n$} & \multicolumn{1}{|c|}{\textnormal{Number of $\mathtt{2}$s}} & \multicolumn{1}{|c|}{ $P(n)$  }& \multicolumn{1}{|c|}{$D(n)$} \\ \hline 
                    1 &                     1 &    &          \\ \hline
   $\mathtt{10^{ 1}}$ &                     5 &  3 & $\mathtt{2.143\cdot 10^{-1}}$ \\ \hline
   $\mathtt{10^{ 2}}$ &                    51 &  6 & $\mathtt{8.333\cdot 10^{-2}}$ \\ \hline
   $\mathtt{10^{ 3}}$ &                   502 &  9 & $\mathtt{2.459\cdot 10^{-2}}$ \\ \hline
   $\mathtt{10^{ 4}}$ &                4\,995 & 11 & $\mathtt{3.318\cdot 10^{-3}}$ \\ \hline
   $\mathtt{10^{ 5}}$ &               49\,999 & 14 & $\mathtt{6.353\cdot 10^{-4}}$ \\ \hline
   $\mathtt{10^{ 6}}$ &              499\,980 & 16 & $\mathtt{8.448\cdot 10^{-5}}$ \\ \hline
   $\mathtt{10^{ 7}}$ &           4\,999\,995 & 19 & $\mathtt{2.464\cdot 10^{-5}}$ \\ \hline
   $\mathtt{10^{ 8}}$ &          50\,000\,202 & 21 & $\mathtt{7.936\cdot 10^{-6}}$ \\ \hline
   $\mathtt{10^{ 9}}$ &         499\,999\,731 & 24 & $\mathtt{3.279\cdot 10^{-6}}$ \\ \hline
   $\mathtt{10^{10}}$ &      5\,000\,005\,565 & 26 & $\mathtt{8.382\cdot 10^{-7}}$ \\ \hline
   $\mathtt{10^{11}}$ &     50\,000\,013\,114 & 29 & $\mathtt{5.606\cdot 10^{-7}}$ \\ \hline
   $\mathtt{10^{12}}$ &    499\,999\,997\,503 & 31 & $\mathtt{1.430\cdot 10^{-7}}$ \\ \hline
   $\mathtt{10^{13}}$ & 4\,999\,999\,971\,938 & 34 & $\mathtt{3.744\cdot 10^{-8}}$ \\ \hline
\end{tabular}}
\caption{\label{table: generalised kolakoski}The output of the calculation of the number of $\mathtt{2}$s in the generalised Kolakoski sequence $K(\mathtt{2},\mathtt{3})$. The column with the number of $\mathtt{2}$s is the sequence \seqnum{A195211} in the On-Line Encyclopedia of Integer Sequences \cite{sloane}.}
\end{center}
\end{table}

\section{Acknowledgement}

The author wishes to thank M. Baake, P. Bugarin, V. Terauds and P. Zeiner at Bielefeld University, Germany, for our discussions of the problem and for reading drafts of the manuscript. This work was supported by the German Research Council (DFG), via CRC 701.

\bigskip
\hrule
\bigskip

\noindent 2010 {\it Mathematics Subject Classification}:
Primary 68R15; Secondary 68Q25.

\noindent \emph{Keywords:} Combinatorics on words, Analysis of algorithms and problem complexity.

\end{document}